\documentclass{article}
\usepackage{amsmath,amssymb,amsfonts}
\usepackage{graphicx}
\usepackage{color}

\usepackage{pb-diagram}
\newtheorem{theorem}{Theorem}

\newtheorem{proposition}{Proposition}
\newtheorem{corollary}{Corollary}
\newtheorem{lemma}{Lemma}

\newtheorem{example}{Example}

\newenvironment{proof}[1][Proof]{\noindent\textit{#1.} }{\hfill$\Box$\medskip}

\title{Systems of the Kowalevski type and discriminantly separable polynomials}
\author{Vladimir Dragovi\'c\footnote{Corresponding author}  (a)  and Katarina Kuki\'c
(b)}

\date{}

\begin{document}

\maketitle

\medskip

\centerline{(a) The Department of Mathematical Sciences, University of Texas at Dallas,}
\centerline{800 West Campbell Road, Richardson TX 75080, USA}
\centerline{Mathematical Institute SANU, Kneza Mihaila 36, 11000 Belgrade, Serbia}
\centerline{e-mail: {\tt Vladimir.Dragovic@utdallas.edu}}

\smallskip

\centerline{(b) Faculty for Traffic and Transport Engineering, University of Belgrade}
\centerline{Vojvode Stepe 305, 11000 Belgrade, Serbia}
\centerline{e-mail: {\tt k.mijailovic@sf.bg.ac.rs}}

\begin{abstract}
Starting from the notion of discriminantly separable polynomials of
degree two in each of three variables, we construct a class of
integrable dynamical systems. These systems can be integrated
explicitly in genus two theta-functions in a procedure which is
similar to the classical one for the Kowalevski top. The
discriminnatly separable polynomials play the role of the Kowalevski
fundamental equation. The natural examples include the Sokolov
systems and the Jurdjevic elasticae.

\medskip

A partir de la notion des polyn\^omes avec les discriminants
s\'eparables qui sont du second degr\'e par rapport à chacune de
trois variables, nous construisons une classe des systemes
dynamiques. On peut r\'esoudre explicitement ces systemes via les
fonctions th\^eta de genre 2, par une proc\'edure similaire à cette
classique pour la toupie de Kowalevski. Les polyn\^omes avec les
discriminants s\'eparables jouent le r\^ole de l'\'equation
fondamentale de Kowalevski. Les exemples naturels incluent les
systemes de Sokolov et les \'elastiques de Jurdjevic.

\end{abstract}

\

\

AMS Subj. Class. 37J35, 37K60 (70E17, 70E40, 39A10)

\

\

Keywords: Integrable systems, Kowalevski top, Discriminantly separable polynomials, Systems of the Kowalevski type

\newpage

\tableofcontents \footnote{Acknowledgements:
 The research was partially
supported by the Serbian Ministry of Science and Technological
Development, Project 174020 "Geometry and Topology of Manifolds,
Classical Mechanics and Integrable Dynamical Systems"}.

\section{Introduction}\label{sec:intro}

\

\subsection{A short note on discriminantly separable polynomials}

\

The Kowalevski top \cite{Kow} is one of the most celebrated
integrable systems. There is a waste literature dedicated to
understanding of Kowalevski original integration procedure, to its
modern versions and hidden symmetries (see for example \cite{Kow1},
\cite{Kot}, \cite{Gol}, \cite{HorMoer}, \cite{Mlo}, \cite{App}, \cite{Del},
\cite{Jur}, \cite{Dub}, \cite{Aud}, \cite{BRST}).

In a recent paper \cite {Drag3} of one of the authors of the present
paper, a new approach to the Kowalevski integration procedure has
been suggested. The novelty has been based on a new notion
introduced therein of {\it discriminantly separable polynomials}. A
family of such polynomials has been constructed there as  pencil
equations from the theory of conics
$$
\mathcal F(w, x_1, x_2)=0,
$$
where $w, x_1, x_2$ are the pencil parameter and the Darboux
coordinates respectively. (For classical applications of the Darboux
coordinates see  Darboux's book \cite {Dar1}, for modern
applications see the book \cite{DragRad} and \cite{Drag2}.) The key
algebraic property of the pencil equation, as quadratic equation in
each of three variables $w, x_1, x_2$ is: {\it all three of its
discriminants are expressed as products of two polynomials in one
variable each}:
\begin{equation}
\aligned \mathcal D_w(\mathcal F)(x_1,x_2)&=P(x_1)P(x_2)\\
\mathcal D_{x_1}(\mathcal F)(w,x_2)&=J(w)P(x_2)\\
\mathcal D_{x_2}(\mathcal F)(w,x_1)&=P(x_1)J(w)
\endaligned
\end{equation}
where $J, P$ are polynomials of degree $3$ and $4$ respectively, and
the elliptic curves
$$\Gamma_1: y^2=P(x), \quad \Gamma_2: y^2=J(s)$$
are isomorphic (see Proposition 1 of \cite{Drag3}) .

In the so-called {\it fundamental Kowalevski equation} (see formula
(\ref{eq:fundKowrel}) below, and also \cite{Kow}, \cite{Kot},
\cite{Gol}) $Q(w, x_1, x_2)=0$, the polynomial $Q(w, x_1, x_2)$
appeares to be an example of a member of the  family, as it was
shown in \cite{Drag3} (Theorem 3). Moreover, all main steps of the
Kowalevski integration now follow as easy and transparent logical
consequences of the theory of  discriminantly separable polynomials.
Let us mention here just one relation, see Corollary 1 from
\cite{Drag3} (known  in the context of the Kowalevski top as {\it
the Kowalevski magic change of variables}):
\begin{equation}\label{eq:kowchange}
\aligned
\frac{dx_1}{\sqrt{P(x_1)}}+\frac{dx_2}{\sqrt{P(x_2)}}&=\frac{dw_1}{\sqrt{J(w_1)}}\\
\frac{dx_1}{\sqrt{P(x_1)}}-\frac{dx_2}{\sqrt{P(x_2)}}&=\frac{dw_2}{\sqrt{J(w_2)}}.
\endaligned
\end{equation}

There is a natural and important question in this context:

{\it Are there other integrable dynamical systems related to
discriminantly separable polynomials?}

\

Referring to this question, we have already constructed discrete
integrable systems related to discriminantly separable polynomials
which are associated to quad-graphs in \cite{DK1}. Now we are
constructing a new class of integrable continuous systems,
generalizing the Kowalevski top. Thus we call the members of that
class -- {\it systems of the Kowalevski type}. A relationship with
the discriminantly separable polynomials gives us possibility to
perform an effective integration procedure, and to provide an
explicit integration formulae in the theta-functions, in general,
associated with genus two curves, as in the original case of
Kowalevski. The first examples of such systems have been constructed
in \cite{DK}. Let us point out here one very important moment
regarding the Kowalevski top and all the systems of Kowalevski type:
the main issue in integration procedures is related to the elliptic
curves $\Gamma_1, \Gamma_2$ and the two-valued groups related to
these elliptic curves, although, as we know, the final part of
integration of the Kowalevski is related to a genus two curve. This,
in a sense unexpected, surprising or possibly not clarified enough
jump in genus from $1$ to $2$ is now explained in our Theorem
\ref{th:sufficient}, and it becomes a trade mark of all systems of the
Kowalevski type.

\

\subsection{An overview of the paper}

\

The paper is organized as follows. In the Section \ref{sec:analysis}
we introduce \textit{systems of the Kowalevski type} of differential
equations by generalizing Kowalevski's considerations. Subsection
\ref{subsec:exKST} presents the Sokolov system from \cite{Sok1} and
\cite{KST} as an example of systems of the Kowalevski type. As a
result, we are providing a full explanation of its integration
procedure. In the third subsection we construct an new example of
such system and in the fourth subsection we give integration
procedure in terms of genus two theta functions and related to them
$P_i, P_{ij}$ functions. Following generalized K\"{o}tter
transformation from \cite{Drag3}, here we reformulate it for
polynomials of degree three. That gives us possibility to integrate
systems in two ways, using properties of $\wp$-function and using
generalized K\"{o}tter transformation. The subsection
\ref{subsec:anothermethod} presents one more method for obtaining
systems of Kowalevski type by studying first integrals and invariant
relations. The Kowalevski top may be seen as a special subcase, and
it serves as a principle
motivating example.\\

In the Section \ref{sec:def1} we consider a simple deformation of
Kowalevski case by using simplest linear gauge transformation on the Kowalevski fundamental equation. As the outcome
we get \textit{the Jurdjevic elasticae} from \cite{Jur}. Systems
analogue to the Jurdjevic elasticae have been obtained before by
Komarov and Kuznetsov (see \cite{Ko} and \cite{KoK}), and they also
serve as motivating examples for our study of systems of Kowalevski
type. We give here the explicit solutions of all these problems in terms of $P_i,P_{ij}$ functions. Another system
 of Kowalevski type is constructed in the Section \ref{sec:def2} following \cite{Kow}
  and by use of a sequence of skilful tricks and identities.\\

Since the main part of this paper is motivated by the Kowalevski top
and by the Kowalevski integration procedure, and since it is the
milestone of the classical integrable systems, we find it useful to
extract some of the key moments of the Kowalevski work \cite{Kow},
here.

\

\subsection{Fundamental steps in the Kowalevski integration procedure}

\

Let us recall briefly  that the Kowalevski top \cite{Kow} is a heavy
spinning top rotating about a fixed point, under the conditions
$I_1=I_2=2I_3,\,I_3=1$, $y_0=z_0=0$. Here $(I_1,I_2,I_3)$ denote the
principal moments of inertia, $(x_0,y_0,z_0)$ is the center of mass,
$c=Mgx_0$, $M$ is the mass of the top, $(p,q,r)$  is the vector of
angular velocity and $(\gamma_1,\gamma_2,\gamma_3)$ are cosines of
the angles between $z$-axis of the fixed coordinate system and the
axes of the coordinate system that is attached to the top and whose
origin coincides with the fixed point. Then the equations of motion
take the following form, see \cite{Kow}, \cite {Gol}:

\begin{equation}\label{eq:Kowsyst}
\aligned
2 \dot p &= qr \qquad \qquad \qquad \dot \gamma_1 =r \gamma_2 - q \gamma_3\\
2 \dot q &=-pr-c\gamma_3 \qquad \quad \dot \gamma_2 =p \gamma_3 - r \gamma_1\\
\dot r &=c \gamma_2 \qquad \qquad \qquad \dot \gamma_3 =q \gamma_1 - p \gamma_2.
\endaligned
\end{equation}

System (\ref{eq:Kowsyst}) has three well known integrals of motion
and a fourth integral
 discovered by Kowalevski
\begin{equation}\label{eq:Kowsystint}
\aligned
&2(p^2+q^2)+r^2 =2c\gamma_1+6l_1\\
&2(p\gamma_1+q\gamma_2)+r\gamma_3 =2l\\
&\gamma_1^2+\gamma_2^2+\gamma_3^2 =1\\
&\left((p+i q)^2+\gamma_1+i\gamma_2\right)\left((p-i q)^2+\gamma_1-i\gamma_2\right)=k^2.
\endaligned
\end{equation}

After the change of variables
\begin{equation}\label{eq:Kowch1}
\aligned x_1 &= p + i q, \quad e_1 = x_1^2 + c(\gamma_1 + i \gamma_2)\\
x_2 &= p - i q, \quad e_2 = x_2^2 + c(\gamma_1 - i
\gamma_2)
\endaligned
\end{equation}
the first integrals (\ref{eq:Kowsystint}) transform into
\begin{equation}\label{eq:Kowinttrans}
\aligned
r^2 &=E+e_1+e_2\\
rc\gamma_3 &=G-x_2e_1-x_1e_2\\
c^2\gamma_3^2 &=F+x_2^2e_1+x_1^2e_2\\
e_1e_2 &=k^2,
\endaligned
\end{equation}
with $E= 6l_1-(x_1+x_2)^2,\, F= 2cl +x_1x_2(x_1+x_2),\,
G= c^2-k^2-x_1^2 x_2^2.$ From the first integrals, one gets
$$(E+e_1+e_2)(F+x_2^2e_1+x_1^2e_2)-(G-x_2e_1-x_1e_2)^2=0$$
which  can be rewritten in the form
\begin{equation}\label{eq:Kowid1}
e_1 P(x_2)+e_2P(x_1)+R_1(x_1,x_2)+k^2(x_1-x_2)^2=0
\end{equation}
where the polynomial $P$ is
$$P(x_i)=x_i^2E+2 x_1F+G =-x_i^4+6l_1 x_i^2 +4lc x_i+c^2-k^2,\, i=1,2$$
and
\begin{equation}
\nonumber
\aligned
&R_1(x_1,x_2)=EG-F^2\\
&=-6l_1x_1^2x_2^2-(c^2-k^2)(x_1+x_2)^2-4lc(x_1+x_2)x_1 x_2+6l_1(c^2-k^2)-4l^2c^2.
\endaligned
\end{equation}
Note that $P$ from the formula above depends only on one variable,
which is not obvious from its definition. Denote
$$R(x_1,x_2)=Ex_1x_2+F(x_1+x_2)+G.$$

>From (\ref{eq:Kowid1}), Kowalevski gets
\begin{equation}\label{eq:Koweq1}
(\sqrt{P(x_1)e_2}\pm\sqrt{P(x_2)e_1})^2=-(x_1-x_2)^2k^2\pm
2k\sqrt{P(x_1)P(x_2)}-R_1(x_1,x_2).
\end{equation}

After a few transformations, (\ref{eq:Koweq1}) can be written in the form
\begin{equation}\label{eq:kott1}
\left[\sqrt{e_1}\frac{\sqrt{P(x_2)}}{x_1-x_2}\pm
\sqrt{e_2}\frac{\sqrt{P(x_1)}}{x_1-x_2}\right]^2=(w_1\pm k)(w_2 \mp
k),
\end{equation}
where $w_1, w_2$ are the solutions of an equation, quadratic in $s$:
\begin{equation}\label{eq:fundKowrel}
Q(s, x_1,x_2)=(x_1-x_2)^2s^2-2R(x_1,x_2)s - R_1(x_1,x_2)=0.
\end{equation}
The quadratic equation (\ref{eq:fundKowrel}) is known as {\bf the
Kowalevski fundamental equation}. The discriminant separability
condition for $Q(s,x_1,x_2)$ is satisfied
$$\mathcal D_{s}(Q)(x_1,x_2)=4P(x_1)P(x_2)$$
$$\mathcal D_{x_1}(Q)(s,x_2)=-8J(s)P(x_2),\,\mathcal D_{x_2}(Q)(s,x_1)=-8J(s)P(x_1)$$
with
$$J(s)=s^3+3l_1s^2+s(c^2-k^2)+3l_1(c^2-k^2)-2l^2c^2.$$
The equations of motion (\ref{eq:Kowsyst}) can be rewritten in new
variables $(x_1, x_2, e_1, e_2, r, \gamma_3)$ in the form:
\begin{equation}\label{eq:Kowsyst2}
\aligned
2 \dot x_1 &= - i f_1, \quad \dot e_1 = -m e_1\\
2 \dot x_2 &= i f_2, \quad \quad \dot e_2 = m e_2.
\endaligned
\end{equation}
There are two additional differential equations for $\dot r$ and
$\dot \gamma_3$. Here $m=i r$ and $ f_1=r x_1 +c \gamma_3,\,
f_2=r x_2 +c \gamma_3.$ One can easily check that

\begin{equation}\label{eq:Kowsyst3}
f_1^2 = P(x_1) + e_1 (x_1-x_2)^2, \quad f_2^2 = P(x_2) + e_2 (x_1-
x_2)^2.
\end{equation}

\

Further integration procedure is described in \cite{Kow}, and in the
Subsection \ref{subsec:exintegration}, we are going to develop
analogue techniques for more general systems in details.

\

\section{Systems of the Kowalevski type}\label{sec:analysis}

\

\subsection{Systems of the Kowalevski type. Definition}\label{subsec:intro}

\

Now, we are going to introduce a class of dynamical systems, which
generalize the Kowalevski top.  Instead of the Kowalevski
fundamental equation (see formula (\ref{eq:fundKowrel})), we start
here from an arbitrary discriminantly separable polynomial of degree
two in each of three variables.

Given a discriminantly separable polynomial of the second degree in
each of three variables
\begin{equation}\label{eq:discriminantsep}
\mathcal{F}(x_1, x_2, s):=A(x_1, x_2)s^2+B(x_1, x_2)s + C(x_1,x_2),
\end{equation}
such that
$$
\mathcal{D}_s(\mathcal{F})(x_1,x_2)=B^2-4AC=4P(x_1)P(x_2),
$$
and
$$
\aligned
\mathcal{D}_{x_1}(\mathcal{F})(s,x_2)&=4P(x_2)J(s)\\
\mathcal{D}_{x_2}(\mathcal{F})(s,x_1)&=4P(x_1)J(s).
\endaligned
$$
Suppose, that a  given system in variables
$x_1,\,x_2,\,e_1,\,e_2,\,r,\,\gamma_3$, after some transformations
reduces to

\begin{equation}\label{eq:analysis}
\aligned 2 \dot x_1 &= - i f_1,\quad \dot e_1 = -m e_1,\\
2 \dot x_2 &= i f_2,\quad \quad \dot e_2 = m e_2,
\endaligned
\end{equation}

where
\begin{equation}\label{eq:analysis1}
f_1^2 = P(x_1) + e_1 A(x_1,x_2), \quad  f_2^2 = P(x_2) + e_2 A(x_1,
x_2).
\end{equation}

Suppose additionally, that the first integrals and invariant relations of the initial system
reduce to a relation
\begin{equation}\label{eq:integral}
 P(x_2)e_1+P(x_1)e_2=C(x_1,x_2)-e_1e_2A(x_1, x_2).
\end{equation}

The equations for $\dot r$ and $\dot \gamma_3$ are not
specified for the moment and $m$ is a function of system's variables.\\

If a system satisfies the above assumptions we will call it {\it a
system of the Kowalevski type}. As it has been pointed out in the
Introduction, see formulae (\ref{eq:Kowid1}, \ref{eq:fundKowrel},
\ref{eq:Kowsyst2},\ref{eq:Kowsyst3}), the Kowalevski top is an
example of the systems of the Kowalevski type.\\

The following theorem is quite general, and concerns all the systems
of the Kowalevski type. It explains in full a subtle mechanism of a
quite miraculous jump in genus, from one to two, in integration
procedure, which has been observed in the Kowalevski top, and now it
is going to be established as a characteristic property of the whole
new class of systems.

\begin{theorem}\label{th:sufficient}
Given a system which reduces to (\ref{eq:analysis},
\ref{eq:analysis1}, \ref{eq:integral}). Then the system is
linearized on the Jacobian of the curve
$$
y^2=J(z)(z-k)(z+k),
$$
where $J$ is a polynomial factor of the discriminant of $\mathcal{F}$ as a
polynomial in $x_1$ and $k$ is a constant such that
$$
e_1e_2=k^2.
$$
\end{theorem}
\begin{proof} Indeed, from the equations of motion on $e_i$ we get
$$
e_1e_2=k^2,
$$
with some constant $k$. Now, we get
$$
\left(\sqrt{e_1}\sqrt{P(x_2)}\pm \sqrt{e_2}\sqrt{P(x_1)}\right)^2=
C(x_1,x_2)-k^2A(x_1, x_2)\pm 2\sqrt{P(x_1)P(x_2)}k.$$ From the last
relations, we get
$$
\left(\sqrt{e_1}\sqrt{\frac{P(x_2)}{A}}+
\sqrt{e_2}\sqrt{\frac{P(x_1)}{A}}\right)^2=(s_1+k)(s_2-k)
$$
and
$$
\left(\sqrt{e_1}\sqrt{\frac{P(x_2)}{A}}-
\sqrt{e_2}\sqrt{\frac{P(x_1)}{A}}\right)^2=(s_1-k)(s_2+k),
$$
where $s_1, s_2$ are the solutions of the quadratic equation
$
\mathcal{F}(x_1, x_2, s)=0
$
in $s$. From the last equations we get
$$
\aligned
2\sqrt{e_1}\sqrt{\frac{P(x_2)}{A}}&=\sqrt{(s_1+k)(s_2-k)}+\sqrt{(s_1-k)(s_2+k)}\\
2\sqrt{e_2}\sqrt{\frac{P(x_1)}{A}}&=\sqrt{(s_1+k)(s_2-k)}-\sqrt{(s_1-k)(s_2+k)}.
\endaligned
$$
Since $s_i$ are solutions of the quadratic equation $F(x_1, x_2,
s_i)=0$, using Vi\`{e}te formulae and discriminant separability
condition, we get
\begin{equation}
s_1+s_2=-\frac{B}{A},\quad s_2-s_1=\frac{\sqrt{4P(x_1)P(x_2)}}{A}.
\end{equation}
>From the last equation, we get
$$
(s_1-s_2)^2=4\frac{P(x_1)P(x_2)}{A^2}.
$$
Using the last equation, we have
$$
\aligned f_1^2&=\frac{P(x_1)}{(s_1-s_2)^2}\left
[(s_1-s_2)^2+4e_1\frac{P(x_2)}{A^2}A\right]\\
&=\frac{P(x_1)}{(s_1-s_2)^2}\left
[(s_1-s_2)^2+\left(\sqrt{(s_1-k)(s_2+k)}+\sqrt{(s_1+k)(s_2-k)}\right)^2\right]\\
&=\frac{P(x_1)}{(s_1-s_2)^2}\left
[\sqrt{(s_1-k)(s_1+k)}+\sqrt{(s_2+k)(s_2-k)}\right]^2.
\endaligned
$$
Similarly
$$
f_2^2=\frac{P(x_2)}{(s_1-s_2)^2}\left
[\sqrt{(s_1-k)(s_1+k)}-\sqrt{(s_2+k)(s_2-k)}\right]^2.
$$
>From the last two equations and from the equations of motion, we get
$$
\aligned
4\dot{x}_1^2&=-\frac{P(x_1)}{(s_1-s_2)^2}\left[ \sqrt{s_1^2-k^2}+\sqrt{s_2^2-k^2}\right]^2\\
4\dot{x}_2^2&=-\frac{P(x_2)}{(s_1-s_2)^2}\left[ \sqrt{s_1^2-k^2}-\sqrt{s_2^2-k^2}\right]^2,
\endaligned
$$
and then
$$
\aligned \frac {d\,x_1}{\sqrt{P(x_1)}}+\frac {d\,x_2}{\sqrt{P(x_2)}}
&=-i\frac{\sqrt{(s_1-k)(s_1+k)}}{s_1-s_2}dt\\
\frac {d\,x_1}{\sqrt{P(x_1)}}-\frac {d\,x_2}{\sqrt{P(x_2)}}
&=-i\frac{\sqrt{(s_2-k)(s_2+k)}}{s_1-s_2}dt.
\endaligned
$$
>From the discriminant separability, one gets (see Corollary 1 from
\cite{Drag3}):
\begin{equation}\label{eq:kowchange}
\aligned
\frac{dx_1}{\sqrt{P(x_1)}}+\frac{dx_2}{\sqrt{P(x_2)}}&=\frac{ds_1}{\sqrt{J(s_1)}}\\
-\frac{dx_1}{\sqrt{P(x_1)}}+\frac{dx_2}{\sqrt{P(x_2)}}&=\frac{ds_2}{\sqrt{J(s_2)}}
\endaligned
\end{equation}
and finally
\begin{equation}\label{eq:kowvareq}
\aligned
\frac{ds_1}{\sqrt{\Phi(s_1)}}+\frac{ds_2}{\sqrt{\Phi(s_2)}}&=0\\
\frac{s_1\,ds_1}{\sqrt{\Phi(s_1)}}+\frac{s_2\,ds_2}{\sqrt{\Phi(s_2)}}&=-i dt,
\endaligned
\end{equation}
where
$$
\Phi(s)=J(s)(s-k)(s+k),
$$
where $\Phi $ is a polynomial of degree up to six.\\

Thus, relations (\ref{eq:kowvareq}) define the Abel map on a curve
$y^2=\Phi(s)$, which has genus $2$ if the roots of the polynomial
$\Phi$ are distinct.
\end{proof}

The last Theorem basically formalizes the original considerations of
Kowalevski, in a slightly more general context of the discriminantly
separable polynomials.

\

We are going to present below the Sokolov system  \cite{Sok1} as an
example of a system of the Kowalevski type, and to provide one new
example of the systems of the Kowalevski type.

\subsection{Example: Sokolov system as a system of the Kowalevski type}\label{subsec:exKST}

\

Sokolov in \cite{Sok1}, \cite{Sok2} considered the Hamiltonian
\begin{equation}\label{eq:HamKST}
\hat{H}=M_1^2+M_2^2+2M_3^2+2c_1\gamma_1+2c_2(\gamma_2 M_3-\gamma_3 M_2)
\end{equation}
on $e(3)$ with the Lie-Poisson brackets
\begin{equation}\label{eq:PBe3}
\{M_i, M_j\}=\epsilon_{ijk}M_k,\quad \{M_i,
\gamma_j\}=\epsilon_{ijk}\gamma_k,\quad \{\gamma_i,
\gamma_j\}=0
\end{equation}
where $\epsilon_{ijk}$ is the totally skew-symetric tensor. In this
section we will prove that this system belongs to the class of systems of the Kowalevski type.

\medskip

The Lie-Poisson bracket (\ref{eq:PBe3}) has two well-known Casimir
functions
$$\aligned
\gamma_1^2+\gamma_2^2+\gamma_3^2&=a,\\
\gamma_1M_1+\gamma_2M_2+\gamma_3M3&=b.
\endaligned
$$

Following \cite{KST} and \cite{Kow} we introduce the new variables
$$z_1=M_1+iM_2, \qquad z_2=M_1-iM_2$$
and
$$
\aligned
e_1&=z_1^2-2c_1(\gamma_1+i\gamma_2)-c_2^2 a-c_2(2\gamma_2 M_3-2\gamma_3 M_2+2i
(\gamma_3 M_1-\gamma_1 M_3)),\\
e_2&=z_2^2-2c_1(\gamma_1-i\gamma_2)-c_2^2 a-c_2(2\gamma_2 M_3-2\gamma_3 M_2+2i
(\gamma_1 M_3-\gamma_3 M_1)).
\endaligned
$$
The second integral of motion for the system (\ref{eq:HamKST})  may
be rewritten as
\begin{equation}\label{eq:KKst}
e_1e_2=k^2.
\end{equation}
The equations of motion in  the new variables $z_i,e_i$ can be
written in the form of (\ref{eq:analysis}) and (\ref{eq:analysis1}),
and this corresponds to the definition of the systems of the
Kowalevski type. It is easy to show that:
$$\dot e_1=-4i M_3 e_1, \qquad \dot{e}_2=4iM_3 e_2$$
and
\begin{equation}\label{eq:z1z2dot}
\aligned
-\dot{z_1}^2&=P(z_1)+e_1(z_1-z_2)^2,\\
-\dot{z_2}^2&=P(z_2)+e_2(z_1-z_2)^2
\endaligned
\end{equation}
where P is a polynomial of fourth degree given by
\begin{equation}\label{eq:PKST}
P(z)= -z^4+2H z^2-8 c_1 b z-k^2+4a c_1^2-2c_2^2(2b^2-Ha)+c_2^4 a.
\end{equation}

In order to demonstrate that the Sokolov system belongs to the class
of the systems of the Kowalevski type, we still need to show that a
relation of the form  (\ref{eq:integral}) is satisfied and we have
to relate it with certain  discriminantly separable polynomial of
the form of (\ref{eq:discriminantsep}).

Starting from the equations
$$\dot{z_1}=-2M_3(M_1-iM_2)+2c_2(\gamma_1 M_2-\gamma_2 M_1)+2c_1\gamma_3$$
and
$$\dot{z_2}=-2M_3(M_1+iM_2)+2c_2(\gamma_1 M_2-\gamma_2 M_1)+2c_1\gamma_3$$
one can get the following

\begin{lemma} The product of the derivatives of the variables $z_i$
is:

\begin{equation}\label{eq:product}
\dot{z_1}\cdot
\dot{z_2}=-\left(F(z_1,z_2)+(H+c_2^2a)(z_1-z_2)^2\right),
\end{equation}
where $F(z_1,z_2)$ is given by:
\begin{equation}\label{eq:FKST1} F(z_1,z_2)=-\frac{1}{2}\left( P(z_1)+P(z_2) +(z_1^2-z_2^2)^2
\right).
\end{equation}
\end{lemma}
 After equating
the square of $\dot{z_1}\dot{z_2}$ from the relation
(\ref{eq:product}) with product $\dot{z_1}^2 \cdot \dot{z_2}^2$ with $\dot{z_i}^2,\,i=1,2,$ from
(\ref{eq:z1z2dot}) we get
\begin{lemma} The variables $z_1, z_2, e_1, e_2$ of the Sokolov
system satisfy the following identity:

\begin{equation}\label{eq:KSTanalysis}
\aligned
&(z_1-z_2)^2[2F(z_1,z_2)(H+c_2^2a)+(z_1-z_2^2)^4(H+c_2^2a)^2-P(z_1)e_2-P(z_2)e_1\\
&-e_1e_2(z_1-z_2)^2]+F^2(z_1,z_2)-P(z_1)P(z_2)=0.
\endaligned
\end{equation}
\end{lemma}
Denote by $C(z_1,z_2)$ a biquadratic polynomial such that
$$F^2(z_1,z_2)-P(z_1)P(z_2)=(z_1-z_2)^2C(z_1,z_2).$$

\begin{proposition} The Sokolov system is a system of the Kowalevski
type. It can be explicitly integrated in the theta-functions of
genus 2.
\end{proposition}

\begin{proof}
We can rewrite relation (\ref{eq:KSTanalysis}) in the form:
\begin{equation}\label{eq:KSTanalysis2}
P(z_1)e_2+P(z_2)e_1=\tilde C(z_1,z_2)-e_1e_2(z_1-z_2)^2,
\end{equation}
with
\begin{equation}\label{eq:CnaseKST}
 \tilde{C}(z_1,z_2)=C(z_1,z_2)+2F(z_1,z_2)(H+c_2^2a)+(H+c_2^2a)^2(z_1-z_2)^2.
\end{equation}
Further integration procedure may be done following Theorem
\ref{th:sufficient}, since the Sokolov system satisfies all the
assumptions of the systems of the Kowalevski type:
(\ref{eq:KSTanalysis}), (\ref{eq:KSTanalysis2}) and
(\ref{eq:z1z2dot}). The discriminantly separable polynomial of three
variables of degree two in each of them  which plays the role of the
Kowalevski fundamental equation in this case is
\begin{equation}\label{eq:KSTnaseF}
 \tilde{F}(z_1,z_2,s)=(z_1-z_2)^2s^2+\tilde B(z_1,z_2)s+\tilde C(z_1,z_2)
\end{equation}
with
$$\tilde B(z_1,z_2)=F(z_1,z_2)+(H+c_2^2 a)(z_1-z_2)^2.$$
The discriminants of (\ref{eq:KSTnaseF}) as  polynomials in $s$ and
in $z_i$, for $i=1,2$ are
$$\mathcal D_{s}(\tilde F)(z_1,z_2)=P(z_1)P(z_2)$$
$$\mathcal D_{z_1}(\tilde F)(s,z_2)=J(s)P(z_2),\,\mathcal D_{z_2}(Q)(s,z_1)=J(s)P(z_1)$$
where $J$ is a polynomial of the third degree
$$\aligned
&J=-8s^3+4(H+3ac_2^2)s^2+(8c_2^2b^2+2k^2-8ac_1^2-8c_2^4a^2-8c_2^2Ha)s-8c_1^2b^2\\
&-4c_2^4ab^2_4c_1^2a^2c_2^2-k^2c_2^2a-Hk^2+2aH^2c_2^2-4Hb^2c_2^2+4Hc_1^2a+4c_2^4Ha^2+2c_2^6a^3.
\endaligned
$$

Finally, as a result of a direct application of Theorem
\ref{th:sufficient}, we get
$$
\aligned
\frac{d \tilde{s_1}}{\sqrt{\Phi(\tilde{s_1})}}+\frac{d \tilde{s_2}}{\sqrt{\Phi(\tilde{s_2})}}&=0\\
\frac{\tilde{s_1}\,d \tilde{s_1}}{\sqrt{\Phi(\tilde{s_1})}}+\frac{\tilde{s_2} \,d \tilde{s_2}}{\sqrt{\Phi(\tilde{s_2})}}&=dt,
\endaligned
$$
where
$$
\Phi(s)=-4J(s)(s-k)(s+k).
$$
\end{proof}

Notice that the formula (\ref{eq:FKST1}) has been introduced in
\cite{KST} together with the variables
\begin{equation}\label{eq:F,sKST}
 s_{1,2}=\frac{F(z_1,z_2)\pm \sqrt{P(z_1)P(z_2)}}{2(z_1-z_2)^2},
\end{equation}
and with a claim that the Sokolov system is equivalent to
$$\dot{s_1}=\frac{\sqrt{P_5(s_1)}}{s_1-s_2}, \quad \dot{s_2}=\frac{\sqrt{P_5(s_2)}}{s_2-s_1}, \quad P_5(s)=P_3(s)P_2(s)$$
with
$$
\aligned
P_3(s)&=s(4s^2+4sH+H^2-k^2+4c_1^2a+2c_2^2(Ha-2b^2)+c_2^4a^2)+4c_1^2b^2,\\
P_2(s)&=4s^2+4(H+c_2^2a)s+H^2-k^2+2c_2^2ha+c_2^4a^2.
\endaligned
$$
Our variables $\tilde s_i$, which are the roots of
(\ref{eq:KSTnaseF}) are related with $s_i$ from \cite{KST} in the
following manner:
$$\tilde {s_i}=s_i+\frac{H+c_2^2 a}{2}.$$ Thus the last Proposition
provides a proof of the claim from \cite{KST}.

\subsection{A new example of an integrable system of the Kowalevski type}\label{subsec:example}

\

Now, we are going to present a new example of a system of the
Kowalevski type. Let us consider the next system of differential
equations:

\begin{equation}\label{eq:ex1}
\aligned
\dot p &= -r q\\
\dot q &= -r p - \gamma_3\\
\dot r &= -2 q (2 p + 1) - 2 \gamma_2\\
\dot \gamma_1 &= 2(q \gamma_3 - r \gamma_2)\\
\dot \gamma_2 &= 2(p \gamma_3 - r \gamma_1)\\
\dot \gamma_3 &= 2(p^2 - q^2)q - 2 q \gamma_1 +2 p \gamma_2.
\endaligned
\end{equation}

\begin{lemma} The system (\ref{eq:ex1}) preserves the standard measure.
\end{lemma}

After a  change of variables
$$\aligned
x_1 &= p + q,\quad e_1 = x_1^2 + \gamma_1 + \gamma_2,\\
x_2 &= p - q, \quad e_2 = x_2^2 + \gamma_1 - \gamma_2,
\endaligned
$$
the system (\ref{eq:ex1}) becomes

\begin{equation}\label{eq:chex1}
\aligned
\dot x_1 &= -r x_1 - \gamma_3 \\
\dot x_2 &= r x_2 + \gamma_3 \\
\dot e_1 &= -2 r e_1 \\
\dot e_2 &=  2 r e_2 \\
\dot r &= - x_1 + x_2 - e_1 + e_2 \\
\dot \gamma_3 &= x_2 e_1 - x_1 e_2.
\endaligned
\end{equation}

The first integrals of the system (\ref{eq:chex1}) can be presented
in the form

\begin{equation}\label{eq:ex1firstint}
\aligned
 r^2 &= 2(x_1 + x_2) +  e_1 +  e_2+h\\
 r \gamma_3 &= -x_1 x_2 -x_2e_1 - x_1e_2 - \frac{g_2}{4}\\
 \gamma_3^2 & = x_2^2e_1 + x_1^2e_2-\frac{g_3}{2}\\
e_1 \cdot e_2 &= k^2.
\endaligned
\end{equation}

>From the integrals (\ref{eq:ex1firstint}) we get a relation of the
form  (\ref{eq:integral})
\begin{equation}\label{eq:ex1-integral}
\aligned &(x_1 - x_2)^2 e_1 e_2 +\left(2 x_1^3+hx_1^2- \frac{g_2}{2} x_1
-\frac{g_3}{2}\right)e_2 + \left(2 x_2^3+hx_2^2- \frac{g_2}{2}
x_2 -\frac{g_3}{2}\right)e_1\\
& -\left(x_1^2 x_2^2 +x_1 x_2 \frac{g_2}{2} + g_3 (x_1 + x_2+\frac{h}{2}) +
\frac{g_2^2}{16}\right) = 0.
\endaligned
\end{equation}

Without loss of generality, we can assume $h=0$ (this can be acheved
by a simple linear change of variables $x_i \mapsto x_i-h/6,\quad s
\mapsto s-h/6$), thus we can use  directly the Weierstrass $\wp$
function. Following the procedure described in Theorem
\ref{th:sufficient} we get
\begin{equation}\label{eq:ex1kowchange}
\aligned
\frac{dx_1}{\sqrt{P(x_1)}}+\frac{dx_2}{\sqrt{P(x_2)}}&=\frac{ds_1}{\sqrt{P(s_1)}}\\
\frac{dx_1}{\sqrt{P(x_1)}}-\frac{dx_2}{\sqrt{P(x_2)}}&=\frac{ds_2}{\sqrt{P(s_2)}}
\endaligned
\end{equation}
where $P(x)$  denotes the polynomial
\begin{equation}\label{eq:polP}
P(x)=2x^3- \frac{g_2}{2} x - \frac{g_3}{2},
\end{equation}
and $s_1,\,s_2$ are the solutions of quadratic equation in $s$:

\begin{equation}\label{eq:F(x_1,x_2,s)}
\aligned  \mathcal{F}(x_1,x_2,s) &:=
A(x_1,x_2)s^2+B(x_1,x_2)s+C(x_1,x_2)\\
&=(x_1-x_2)^2 s^2 + \left(-2 x_1 x_2
(x_1+x_2)+\frac{g_2}{2}(x_1+x_2)+g_3\right)s\\
& + x_1^2 x_2^2 +x_1 x_2 \frac{g_2}{2} + g_3 (x_1 + x_2) +
\frac{g_2^2}{16}=0.
\endaligned
\end{equation}

Finally, we get
\medskip
\begin{corollary}
The system of differential equations (\ref{eq:ex1}) is integrated
through the solutions of the system
\begin{equation}\label{eq:ex1kowvareq}
\aligned
\frac{ds_1}{\sqrt{\Phi(s_1)}}+\frac{ds_2}{\sqrt{\Phi(s_2)}}&=0\\
\frac{s_1\,ds_1}{\sqrt{\Phi(s_1)}}+\frac{s_2\,ds_2}{\sqrt{\Phi(s_2)}}&=2\,dt,
\endaligned
\end{equation}
where $ \Phi(s)=P(s)(s-k)(s+k).$
\end{corollary}

\subsection{Explicit integration in genus two theta functions}\label{subsec:exintegration}

\

This subsection is devoted to an explicit integration of the system
(\ref{eq:ex1}). The integration of the system will be realized in
two ways. The first one is based more directly on the Kowalevski
original approach \cite{Kow} and uses the properties of the elliptic
functions. The second one follows K\"{o}tter's approach (see
\cite{Kot} and Golubev \cite{Gol}). A generalization of the
K\"{o}tter transformation was derived in \cite{Drag3} for a
polynomial $P(x)$ of degree four. Here we will reformulate such a
transformation for $P(x)$ of degree three.

\medskip

We are going to consider here, as in \cite{Kow}, the case where the
zeros $l_i,\,i=1,\,2,\,3$ of the polynomial $P$ of degree three are
real and $l_1>l_2>l_3$. Denote
$$l=(l_1-l_2)(l_2-l_3)(l_3-l_1).$$

Following Kowalevski, we consider the functions
\begin{equation}\label{eq:pi}
P_i=\sqrt{(s_1-l_i)(s_2-l_i)}, \quad i=1,2,3
\end{equation}
 and
\begin{equation}\label{eq:pij}
P_{ij}=P_i P_j \left(\frac{\dot s_1}{(s_1-l_i)(s_1-l_j)}+\frac{\dot
s_2}{(s_2-l_i)(s_2-l_j)}\right).
\end{equation}

Then by simple calculations one gets
\begin{equation}\label{eq:ex1-rel P_i}
\aligned \dot P_1 &= \frac{P_3 P_{13}-P_2P_{12}}{2(l_2-l_3)},\quad
\dot P_2 = \frac{P_1 P_{12}-P_3P_{23}}{2(l_3-l_1)},\\
\dot P_3 &= \frac{P_2 P_{23}-P_1P_{13}}{2(l_1-l_2)},\quad \dot
P_{ij} =\frac{1}{2}P_iP_j.
\endaligned
\end{equation}

We will now derive the expressions for
$p,\,q,\,r,\,\gamma_1,\,\gamma_2,\,\gamma_3$ in terms of
$P_i,\,P_{ij}$ functions for $i,\,j=1,\,2,\,3$.

\medskip

Denote by
$$du_i=\frac{dx_i}{\sqrt{4x_i^3-g_2x_i-g_3}},\quad i=1,2.$$
Then $x_i=\wp (u_i),$ and we get $s_1=\wp(u_1+u_2),\,
s_2=\wp(u_1-u_2).$\\

We will use the following properties of $\wp$-function, see
\cite{Kow}:
\begin{equation}\label{eq:wp function}
\aligned
\wp(u_1)+\wp(u_2) &= -2\frac{(l_2^2-l_3^2)P_1+(l_3^2-l_1^2)P_2+(l_1^2-l_2^2)P_3}{(l_2-l_3)P_1+(l_3-l_1)P_2+(l_1-l_2)P_3},\\
\wp(u_1)-\wp(u_2) &= \frac{-2l}{(l_2-l_3)P_1+(l_3-l_1)P_2+(l_1-l_2)P_3},\\
\wp(u_1) \cdot \wp(u_2) &=
-\big[\frac{(l_2-l_3)(l_1^2+l_2l_3)P_1+(l_3-l_1)(l_2^2+l_1l_3)P_2}{(l_2-l_3)P_1+(l_3-l_1)P_2+(l_1-l_2)P_3}\\
&+\frac{(l_1-l_2)(l_3^2+l_1l_2)P_3}{(l_2-l_3)P_1+(l_3-l_1)P_2+(l_1-l_2)P_3}\big].
\endaligned
\end{equation}

After some calculations, we get  for the variables
$p,\,q,\,r,\,\gamma_1,\,\gamma_2,\,\gamma_3$ the expressions in
terms of $P_i$ and $P_{ij}$ -functions for $i,j=1,2,3$:

\begin{equation}\label{eq:ex1Kowsolp}
p=
\frac{x_1+x_2}{2}=\frac{\wp(u_1)+\wp(u_2)}{2}=-\frac{(l_2^2-l_3^2)P_1+(l_3^2-l_1^2)P_2+(l_1^2-l_2^2)P_3}{(l_2-l_3)P_1+(l_3-l_1)P_2+(l_1-l_2)P_3},
\end{equation}
\begin{equation}\label{eq:ex1Kowsolq}
q=\frac{x_1-x_2}{2}=\frac{\wp(u_1)-\wp(u_2)}{2}=-\frac{l}{(l_2-l_3)P_1+(l_3-l_1)P_2+(l_1-l_2)P_3},
\end{equation}
\begin{equation}\label{eq:ex1Kowsolr}
r=-\frac{\dot p}{q}=\frac{1}{2}\frac{(l_1-l_2)P_{12}+(l_2-l_3)P_{23}+(l_3-l_1)P_{13}}{(l_2-l_3)P_1+(l_3-l_1)P_2+(l_1-l_2)P_3},
\end{equation}

\begin{equation}\label{eq:ex1Kowsolgamma1}
\aligned
\gamma_1 &=\frac{\left((l_1-l_2)P_{12}+(l_2-l_3)P_{23}+(l_3-l_1)P_{13}\right)^2}{8\left((l_2-l_3)P_1+(l_3-l_1)P_2+(l_1-l_2)P_3\right)^2}\\
&-\frac{\left((l_2^2-l_3^2)P_1+(l_3^2-l_1^2)P_2+(l_1^2-l_2^2)P_3\right)^2+l^2}{\left((l_2-l_3)P_1+(l_3-l_1)P_2+(l_1-l_2)P_3\right)^2}\\
&+2\frac{(l_2^2-l_3^2)P_1+(l_3^2-l_1^2)P_2+(l_1^2-l_2^2)P_3}{(l_2-l_3)P_1+(l_3-l_1)P_2+(l_1-l_2)P_3},
\endaligned
\end{equation}
\begin{equation}\label{eq:ex1Kowsolgamma2}
\aligned \gamma_2
& =\frac{l}{\left((l_2-l_3)P_1+(l_3-l_1)P_2+(l_1-l_2)P_3\right)^2}\\
&\cdot\big[
(l_2-l_3-2l_2^2+2l_3^2)P_1+(l_3-l_1-2l_3^2+2l_1^2)P_2\\
&+(l_1-l_2-2l_1^2+2l_2^2)P_3\big]\\
&-\frac{\left((l_2-l_3)P_2P_3+(l_3-l_1)P_1P_3+(l_1-l_2)P_1P_2\right)}{8(l_2-l_3)P_1+(l_3-l_1)P_2+(l_1-l_2)P_3}\\
&-\frac{
(l_2-l_3)P_{23}+(l_3-l_1)P_{13}+(l_1-l_2)P_{12}}{8\left((l_2-l_3)P_1+(l_3-l_1)P_2+(l_1-l_2)P_3\right)^2}\\
&\cdot\left(P_3P_{13}+P_1P_{12}+P_2P_{23}-P_2P_{12}-P_1P_{13}-P_3P_{23}\right),
\endaligned
\end{equation}

\begin{equation}\label{eq:ex1Kowsolgamma3}
\aligned \gamma_3
&=\frac{1}{2}\frac{(l_2^2-l_3^2)P_1+(l_3^2-l_1^2)P_2+(l_1^2-l_2^2)P_3}{\left((l_2-l_3)P_1+(l_3-l_1)P_2+(l_1-l_2)P_3\right)^2}\\
&\cdot
\left((l_1-l_2)P_{12}+(l_2-l_3)P_{23}+(l_3-l_1)P_{13}\right)\\
&-\frac{l}{2}\frac{P_3P_{13}+P_1P_{12}+P_2P_{23}-P_2P_{12}-P_1P_{13}-P_3P_{23}}{\left((l_2-l_3)P_1+(l_3-l_1)P_2+(l_1-l_2)P_3\right)^2}.
\endaligned
\end{equation}

\medskip

The expressions for $P_i$ and $P_{ij}$ for $i,j=1,2,3$  in terms of
the theta-functions  are given in \cite{Kow}, \cite{baker}.

\medskip

Now, we will perform integration following K\"{o}tter \cite{Kot} and
Golubev \cite{Gol}. First, we will formulate an extension of the
K\"{o}tter's transformation for a degree three polynomial $P(x)=2
x^3-\frac{g_2}{2}x -\frac{g_3}{2}$.

\begin{proposition}
For a polynomial $\mathcal{F}(x_1,x_2,s)$ given by the formula
(\ref{eq:F(x_1,x_2,s)}),
 there exist polynomials $\alpha(x_1,x_2,s)$,
$\beta(x_1,x_2,s)$, $P(s)$ such that the following identity
\begin{equation}\label{eq:kotter}
\mathcal{F}(x_1,x_2,s)=\alpha^2(x_1,x_2,s)+P(s)\beta(x_1,x_2,s),
\end{equation}
is satisfied. The polynomials are defined by the formulae:
  $\alpha(x_1,x_2,s) = 2 s^2+s( x_1 + x_2 )-x_1 x_2
-g_2/4$,
  $\beta(x_1,x_2,s) = -2(x_1+x_2+s)$, and  $P(s) = 2 s^3-g_2s/2 -g_3/2$,
where $P$ coincides with the polynomial from formula
(\ref{eq:polP}).
\end{proposition}

The proof follows by a direct calculation.

\medskip

Define $\mathcal{\hat F}(s)=\mathcal{F}(x_1,x_2,s)/(x_1-x_2)^2,$ and
consider the identity $\mathcal{\hat
F}(s)=(s-u)^2+(s-u)\mathcal{\hat F}'(u)+ \mathcal{\hat F}(u).$ Then,
from (\ref{eq:kotter}) we get
$$
\aligned
&(s-u)^2(x_1-x_2)^2+2(s-u)\left(u(x_1-x_2)^2+\frac{B(x_1,x_2)}{2}\right)\\
&+\alpha^2(x_1,x_2,u)+P(u)\beta(x_1,x_2,u)=0.
\endaligned
$$

\medskip
\begin{corollary}\label{cor:kotterP_i}
(a) The solutions $s_1, s_2$ of the last equation in $s$ satisfy the
following identity in $u$:
$$(s_1-u)(s_2-u)=\frac{\alpha^2(x_1,x_2,u)}{(x_1-x_2)^2}+P(u) \frac{\beta(x_1,x_2,u)}{(x_1-x_2)^2},$$
where $P(u)$ is the polynomial defined with (\ref{eq:polP}).\\

(b) Functions $P_i$ satisfy
\begin{equation}\label{eq:kotter P_i}
 P_i = \frac{\alpha(x_1,x_2,l_i)}{x_1-x_2}=\left(2l_i^2-\frac{g_2}{4}\right)\frac{1}{x_1-x_2}+l_i\frac{x_1+x_2}{x_1-x_2}
 -\frac{x_1x_2}{x_1-x_2}.
\end{equation}
\end{corollary}

\medskip

Now we introduce a more convenient notation
$$X=\frac{x_1x_2}{x_1-x_2},\quad Y=\frac{1}{x_1-x_2},\quad Z=\frac{x_1+x_2}{x_1-x_2}.$$

\medskip
\begin{lemma}
The quantities $X,Y,Z$ satisfy the system of linear equations
\begin{equation}\label{eq:X,Y,Z-system}
\aligned -X+\left(2l_1^2-\frac{g_2}{4}\right)Y+ l_1 Z &= P_1\\
-X+\left(2l_2^2-\frac{g_2}{4}\right)Y+ l_2 Z &= P_2\\
-X+\left(2l_3^2-\frac{g_2}{4}\right)Y+ l_3 Z &= P_3.
\endaligned
\end{equation}
The solutions of the system (\ref{eq:X,Y,Z-system}) are
$$
\aligned
X &=-\frac{(g_2+8l_2l_3)(l_2-l_3)P_1+(g_2+8l_3l_1)(l_3-l_1)P_2}{8l}-\frac{(g_2+8l_1l_2)(l_1-l_2)P_3}{8l},\\
Y &= \frac{(l_2-l_3)P_1+(l_3-l_1)P_2+(l_1-l_2)P_3}{-2l},\\
Z &= \frac{(l_2^2-l_3^2)P_1+(l_3^2-l_1^2)P_2+(l_1^2-l_2^2)P_3}{l}.
\endaligned
$$
\end{lemma}

\medskip

Using Vi\`{e}te formulae for polynomial $P(x)$ we can rewrite $X$ in
the form
$$X=\frac{(l_2-l_3)(l_1^2+l_2l_3)P_1+(l_3-l_1)(l_2^2+l_1l_3)P_2+(l_1-l_2)(l_3^2+l_1l_2)P_3}{2l}.$$

Now, from the expressions for $X,\,Y,\,Z$ we get
$q=(x_1-x_2)/2=1/(2Y),$ $p=(x_1+x_2)/2=Z/(2Y).$

The expressions for $r$ and $\gamma_i,\,i=1,2,3$ can now be derived
in terms of $P_i,P_{ij}$-functions from the equations of the system
(\ref{eq:ex1}), see formulae
(\ref{eq:ex1Kowsolp})-(\ref{eq:ex1Kowsolgamma3}).

\medskip

\subsection{Another method for obtaining systems of Kowalevski type}\label{subsec:anothermethod}

\ In \cite{DK} a method for constructing examples of systems of the
Kowalevski type has been presented.

Now, we will show another method for construction of  systems of Kowalevski type, that reduces to (\ref{eq:analysis}),
(\ref{eq:analysis1}), (\ref{eq:integral}), with possible first
integrals of the form

\begin{equation}\label{eq:firstint}
\aligned
 r^2 &= E + p_2 e_1 + p_1 e_2\\
 r \gamma_3 &= F-q_2e_1 - q_1e_2\\
  \gamma_3^2& = G +r_2e_1 + r_1e_2\\
e_1 \cdot e_2 &= k^2.
\endaligned
\end{equation}

We search for functions $E, F, G, p_i, q_i, r_i$, $i=1,2$ starting
from
$$
(E + p_2 e_1 + p_1 e_2)(G +r_2e_1 + r_1e_2)-(F-q_2e_1 - q_1e_2)^2=0.
$$
We want to end up with a relation of the form (\ref{eq:integral}).
One set of conditions is the annulation of the coefficients with
$e_1^2, e_2^2$:

\begin{equation}\label{eq:qgen}
 p_1r_1 = q_1^2,\, \, \, p_2r_2 = q_2^2.
\end{equation}
We come to the relation
\begin{equation}\label{eq:integral1}
 \tilde{P}_2e_1+\tilde{P}_1e_2=C(x_1,x_2)-e_1e_2A(x_1, x_2)
\end{equation}
where

\begin{equation}
 A = p_1r_2 + p_2r_1 - 2 q_1q_2, \quad
  C = F^2 - EG,
\end{equation}
and
\begin{equation}\label{eq:Pgen}
\tilde{P}_1 = r_1E + 2q_1F + p_1G, \quad
  \tilde{P}_2 = r_2E + 2q_2F + p_2G.
\end{equation}
Let us assume:

\begin{equation}\label{eq:firstguess}
B_1=E\sqrt{r_1r_2} + F(\sqrt{p_1r_2}+\sqrt{p_2r_1}) +
G\sqrt{p_1p_2}.
\end{equation}
Then, we have the following
\begin{lemma}\label{lemma:factorization} The functions $A, B_1, C,
\tilde{P}_1, \tilde{P}_2$ defined above, satisfy the identity:
$$
B_1^2-AC = \tilde{P}_1 \tilde{P}_2.
$$
\end{lemma}

\begin{lemma}

For a system which satisfies the relations (\ref{eq:firstint}) and
(\ref{eq:qgen})-(\ref{eq:Pgen}), the functions $f_i$ defined by
$$ f_i=\sqrt{r_i} r+\sqrt{p_i} \gamma_3, \quad i=1,\,2$$
satisfy the assumption (\ref{eq:analysis1}).
\end{lemma}
\begin{proof} Proof is done by a straightforward calculation.
\begin{equation}
\aligned f_1^2 &=r_1 r^2+2 \sqrt{p_1 r_1} r \gamma_3 +p_1
\gamma_3^2\\
&= r_1 (E+p_2 e_1+p_1 e_2)+ 2q_1 (F-q_2 e_1-q_1 e_2)+p_1 (G+ r_2
e_1+r_1 e_2)\\
&= Er_1 +Gp_1 +2q_1 F +e_1(r_1 p_2+p_1
r_2-2q_1q_2)+e_2(r_1p_1+p_1r_1-2q_1^2)\\
 &= \tilde{P}_1 +e_1 (r_1 p_2+p_1
r_2-2\sqrt{r_1p_2r_2p_1})= \tilde{P}_1 +e_1 A.
\endaligned
\end{equation}
In the same way, we get $f_2^2 = \tilde{P}_2 +e_2 A.$
\end{proof}
\medskip

Now, let us introduce the second assumption:

\begin{equation}\label{eq:pqr}
 r_i  = x_i^2,\, p_i  = (x_i-a)^2,\,
q_i  = x_i(x_i-a), \, i=1,2.
\end{equation}

In order to get a relation of the form (\ref{eq:integral}) from
(\ref{eq:integral1}), the last, crucial condition, is $
 \tilde{P}_1  = P(x_1),\,\,  \tilde{P}_2  = P(x_2),$
which means that the functions $\tilde{P}_1$ and $\tilde{P}_2$ are
equal to {\bf a polynomial $P$ of one variable, $x_1$ in the former
and $x_2$ in the latter case}.

In order to satisfy the last requirement, we need to guess a right
form of the functions $E, F, G$.  Let us, finally, assume:

\begin{equation}\label{eq:EFG}
\aligned
 E &= -2(x_1-a)^2(x_2-a)^2+C_1\\
 F &= x_1(x_1-a)(x_2-a)^2 + x_2(x_2-a)(x_1-a)^2 + C_2\\
 G &= -2 x_1x_2(x_2-a)(x_1-a) + C_3.
\endaligned
\end{equation}

\begin{theorem}\label{th:ex} The polynomials $E, F, G, p_i,
 q_i, r_i$ define a  completely integrable system, with
 \begin{equation}\label{th:ex-function}
 f_i = x_i r + (x_i-a) \gamma_3, \quad i=1,2.
 \end{equation}
 The system is explicitly integrated on a pinched genus-two curve
 in the theta-functions of the generalized Jacobian of an elliptic curve.
 \end{theorem}

\medskip

\medskip
Similar systems appeared in a  slightly different context in the
works of Appel'rot, Mlodzeevskii, Delone in their study of
degenerations of the Kowalevski top (see \cite {App},
\cite{Mlo},\cite {Del}). In particular, we may construct {\it
Delone-type} solutions of the last system: $ s_1=0, \,
s_2=\wp(i(t-t_0)/4). $
\medskip

\begin{proposition}\label{prop:eq-th1}The equations of motion for the completely integrable system
described in Theorem \ref{th:ex}, with the functions $f_i,\,i=1,2$
defined by (\ref{th:ex-function}) are
\begin{equation}
\nonumber \aligned \dot x_1 &=-\frac{i}{2}f_1, \qquad \dot e_1 = -me_1,\\
\dot x_2 &=\frac{i}{2}f_2, \qquad \dot e_2 = me_2,\\
\dot r &= (x_2-x_1)(x_2-a)(x_1-a)ai+\frac{i f_2(x_2-a)}{2r}e_1
-\frac{i f_1(x_1-a)}{2r}e_2\\
&-\frac{e_1(x_2-a)^2-e_2(x_1-a)^2}{2r}m,\\ \dot \gamma_3 &=
\frac{e_2x_1^2-e_1x_2^2}{2\gamma_3}m-\frac{i f_2x_2e_1}{2\gamma_3}-\frac{i f_1x_1e_2}{2\gamma_3}\\
&-\frac{i(\gamma_3(a-x_1)(a-x_2)-rx_1x_2)(x_2-x_1)a}{2\gamma_3},
\endaligned
\end{equation}
with
\begin{equation}\label{eq:th2m}
 m=i \frac{-(r+\gamma_3)f_1^2e_2
-(\gamma_3 a(r+\gamma_3)-x_2(\gamma_3^2+2r\gamma_3-r^2))
f_2e_1-ar(x_1-x_2)f_1f_2
}{f_2^2e_1-f_1^2e_2}.
\end{equation}

\end{proposition}

\begin{proof}
The system of equations described in Theorem \ref{th:ex} is
\begin{equation}
\nonumber \aligned \dot x_1 &= -\frac{i}{2}(x_1
r+(x_1-a)\gamma_3), \quad  \dot e_1 = -me_1\\
\dot x_2 &= \frac{i}{2}(x_2 r+(x_2-a)\gamma_3), \quad \dot e_2
= me_2,
\endaligned
\end{equation}
with the first integrals
\begin{equation}\label{eq:th2r^2}
r^2 =-2(x_1-a)^2(x_2-a)^2+C_1+(x_2-a)^2 e_1+(x_1-a)^2 e_2,
\end{equation}
\begin{equation}\label{eq:th2rg}
r \gamma_3 = x_1(x_1-a)(x_2-a)^2 + x_2(x_2-a)(x_1-a)^2+C_2
-x_2(x_2-a)e_1-x_1(x_1-a)e_2,
\end{equation}
\begin{equation}\label{eq:th2g^2}
\gamma_3^2 = -2 x_1x_2(x_2-a)(x_1-a) + C_3 +x_2^2 e_1 +x_1^2 e_2,
\end{equation}
$$e_1e_2=k^2.$$

By differentiating the first integrals (\ref{eq:th2r^2}) and
(\ref{eq:th2g^2}) we get equations for $\dot r$ and $\dot \gamma_3$.
Then, by differentiating the first integral (\ref{eq:th2rg}) we get
\begin{equation}
\nonumber \aligned &r\dot \gamma_3+\dot r \gamma_3 =\dot x_1
(x_1-a)(x_2-a)^2+x_1 \dot x_1(x_2-a)^2+x_1(x_1-a)2(x_2-a)\dot
x_2\\
&+\dot x_2(x_2-a)(x_1-a)^2+x_2 \dot
x_2(x_1-a)^2+x_2(x_2-a)2(x_1-a)\dot x_1-x_1\dot x_1 e_2\\
&-\dot x_1(x_1-a)e_2-mx_1(x_1-a)e_2-\dot
x_2(x_2-a)e_1-x_2\dot x_2e_1+m x_2(x_2-a)e_1.
\endaligned
\end{equation}
By plugging the obtained expressions  for $\dot r$ and $\dot
\gamma_3$ we get an equation for $m$ with the solution
(\ref{eq:th2m}).
\end{proof}

\medskip

\begin{proposition}
The system of differential equations defined by Proposition
\ref{prop:eq-th1} is integrated  through the solutions of the system
\begin{equation}\label{eq:sysp22}
\aligned \frac {ds_1}{s_1\sqrt{\Phi_1(s_1)}} +\frac
{ds_2}{s_2\sqrt{\Phi_1(s_2)}}&=0\\
\frac {ds_1}{\sqrt{\Phi_1(s_1)}} +\frac
{ds_2}{\sqrt{\Phi_1(s_2)}}&=\frac{i}{2}dt,
\endaligned
\end{equation}
where $\Phi_1(s)=s(s-e_4)(s-e_5)$ is a polynomial of degree 3.
\end{proposition}

\medskip

By choosing expressions for $p_i,\,r_i, \, i=1,2$ and then finding
$E,\,F,\,G$ such that $\tilde{P}_i=Er_i+2Fq_i+Gp_i$ be polynomials,
we can obtain new examples of systems of the Kowalevski type.

\medskip

\subsection{Another class of system of Kowalevski type}\label{sec:def2}

\

In this section we will consider another class of systems of
Kowalevski type. We consider a situation analogue to that from the
beginning of the Section \ref{subsec:intro}. The only difference is
that the systems we are going to consider now,  reduce to
(\ref{eq:analysis}), where
\begin{equation}\label{eq:analysis2}
\aligned
f_1^2&=P(x_1)-\frac{C}{e_2}\\
f_2^2&=P(x_2)-\frac{C}{e_1}.
\endaligned
\end{equation}
The next Proposition is an analogue of Theorem \ref{th:sufficient}.
Thus, the new class of systems also has a striking property of
jumping genus in integration procedure.

\medskip

\begin{proposition}\label{th:sufficient1}
Given a system which reduces to (\ref{eq:analysis}), where
\begin{equation}
\aligned
f_1^2&=P(x_1)-\frac{C}{e_2}\\
f_2^2&=P(x_2)-\frac{C}{e_1}
\endaligned
\end{equation}
and integrals reduce to (\ref{eq:integral}); $A, C, P$ form a
discriminantly separable polynomial $\mathcal{F}$ given with
(\ref{eq:discriminantsep}). Then the system is linearized on the
Jacobian of the curve
$$
y^2=J(z)(z-k)(z+k),
$$
where $J$ is a polynomial factor of the discriminant of
$\mathcal{F}$ as a polynomial in $x_1$ and $k$ is a constant such
that
$$
e_1e_2=k^2.
$$
\end{proposition}

Although the proof is a variation of the proof of the Theorem
\ref{th:sufficient}, there are some interesting steps and algebraic
transformations we would like to point out.

\begin{proof}  In the same manner as in
Theorem \ref{th:sufficient}, we obtain
$$
\aligned \left(\sqrt{e_1}\sqrt{\frac{P(x_2)}{A}}+
\sqrt{e_2}\sqrt{\frac{P(x_1)}{A}}\right)^2&=(s_1+k)(s_2-k)\\
\left(\sqrt{e_1}\sqrt{\frac{P(x_2)}{A}}-
\sqrt{e_2}\sqrt{\frac{P(x_1)}{A}}\right)^2&=(s_1-k)(s_2+k),
\endaligned
$$
where $s_1, s_2$ are the solutions of the quadratic equation
$$
\mathcal{F}(x_1, x_2, s)=0
$$
in $s$. From the last equations, dividing by $k=\sqrt{e_1e_2}$, we
get
$$
\aligned
2\sqrt{\frac{P(x_2)}{e_2A}}&=\frac{1}{k}\left(\sqrt{(s_1+k)(s_2-k)}+\sqrt{(s_1-k)(s_2+k)}\right)\\
2\sqrt{\frac{P(x_1)}{e_1A}}&=\frac{1}{k}\left(\sqrt{(s_1+k)(s_2-k)}-\sqrt{(s_1-k)(s_2+k)}\right).
\endaligned
$$
Using $\displaystyle (s_1-s_2)^2=4\frac{P(x_1)P(x_2)}{A^2}
$, we get
$$
\aligned f_1^2&=P(x_1)-\frac{C(x_1,x_2)}{e_2}=\frac{(s_1-s_2)^2A^2}{4P(x_2)}-\frac{C}{e_2}=\frac{A^2}{4P(x_2)}\left[(s_1-s_2)^2 -\frac{C}{A}\frac{4P(x_2)}{e_2A}\right]\\
&=\frac{P(x_1)}{(s_1-s_2)^2}\left[ (s_1-s_2)^2-s_1s_2\frac{1}{k^2}\left(\sqrt{(s_1+k)(s_2-k)}+\sqrt{(s_1-k)(s_2+k)} \right)^2 \right]\\
&=\frac{P(x_1)}{(s_1-s_2)^2}\left[s_1^2-2s_1s_2+s_2^2-\frac{2s_1s_2}{k^2}\left(s_1s_2-k^2+\sqrt{(s_1^2-k^2)(s_2^2-k^2)} \right)  \right]\\
&=\frac{P(x_1)}{k^2(s_1-s_2)^2}\left[k^2(s_1^2+s_2^2)-2s_1^2s_2^2-2s_1s_2\sqrt{(s_1^2-k^2)(s_2^2-k^2)} \right]\\
&=-\frac{P(x_1)}{k^2(s_1-s_2)^2}\left[
s_2\sqrt{s_1^2-k^2}+s_1\sqrt{s_2^2-k^2}\right]^2.
\endaligned
$$
Similarly
$$
f_2^2=-\frac{P(x_2)}{k^2(s_1-s_2)^2}\left[
s_2\sqrt{s_1^2-k^2}-s_1\sqrt{s_2^2-k^2}\right]^2.
$$
>From the last two equations and from the equations of motion, we get
$$
\aligned
2\dot{x}_1&=\frac{-i\sqrt{P(x_1)}}{k(s_1-s_2)}\left[s_2\sqrt{s_1^2-k^2}+s_1\sqrt{s_2^2-k^2}\right]\\
2\dot{x}_2&=\frac{-i\sqrt{P(x_2)}}{k(s_1-s_2)}\left[s_2\sqrt{s_1^2-k^2}-s_1\sqrt{s_2^2-k^2}\right],
\endaligned
$$
and
$$
\aligned \frac {d\,x_1}{\sqrt{P(x_1)}}+\frac {d\,x_2}{\sqrt{P(x_2)}}
&=\frac{-i s_2\sqrt{s_1^2-k^2}}{k(s_1-s_2)}dt\\
\frac {d\,x_1}{\sqrt{P(x_1)}}-\frac {d\,x_2}{\sqrt{P(x_2)}}
&=\frac{-i s_1\sqrt{s_2^2-k^2}}{k(s_1-s_2)}dt.
\endaligned
$$
Discriminant separability condition (see Corollary 1 from
\cite{Drag3} ) gives
\begin{equation}\label{eq:kowchange1}
\aligned
\frac{dx_1}{\sqrt{P(x_1)}}+\frac{dx_2}{\sqrt{P(x_2)}}&=\frac{ds_1}{\sqrt{J(s_1)}}\\
\frac{dx_1}{\sqrt{P(x_1)}}-\frac{dx_2}{\sqrt{P(x_2)}}&=-\frac{ds_2}{\sqrt{J(s_2)}}.
\endaligned
\end{equation}
Finally
\begin{equation}\label{eq:kowvareq1}
\aligned
\frac{ds_1}{\sqrt{\Phi(s_1)}}+\frac{ds_2}{\sqrt{\Phi(s_2)}}&=\frac{i}{k}d\,t\\
\frac{s_1\,ds_1}{\sqrt{\Phi(s_1)}}+\frac{s_2\,ds_2}{\sqrt{\Phi(s_2)}}&=0,
\endaligned
\end{equation}
where
$$
\Phi(s)=J(s)(s-k)(s+k),
$$
is a polynomial of degree up to six.
\end{proof}

\section{A deformation of the Kowalevski top}\label{sec:def1}

\

In this Section we are going to derive the explicit solutions in
genus two theta-functions of the Jurdjevic elasticae \cite{Jur} and
for similar systems \cite {Ko}, \cite{KoK}. First, we show that we
can get the elasticae from the Kowalevski top by using the simplest
gauge transformations of the discriminantly separable polynomials.\\

Consider a discriminantly separable polynomial

$$\mathcal{F}(x_1,x_2,s):=s^2A+sB+C$$
where

\begin{equation}\label{eq:deformationB}
 A=(x_1-x_2)^2,\, B=-2(Ex_1x_2+F(x_1+x_2)+G),\,
C=F^2-EG.
\end{equation}

A simple affine gauge transformation
$
s\mapsto t + \alpha
$
transforms $\mathcal{F}(x_1,x_2,s)$ into

$$
\mathcal{F}_{\alpha}(x_1, x_2, t)=t^2 A_{\alpha} + t B_{\alpha} +C_{\alpha}
,
$$
with
\begin{equation}
 A_{\alpha}=A,\, B_{\alpha}=B+2\alpha A,\, C_{\alpha}=C+\alpha B +\alpha ^2A.
\end{equation}
Next, we denote $F_{\alpha}=F+\alpha F_1,\, E_{\alpha}=E+\alpha
E_1,\, G_{\alpha}=G+\alpha G_1.$ From
$$C_{\alpha}=F_{\alpha}^2-E_{\alpha}G_{\alpha},$$ by equating powers
of $\alpha$, we get
\begin{equation} B=2FF_1-E_1G-EG_1,\quad
A=F_1^2-E_1G_1.
\end{equation}
>From (\ref{eq:deformationB}) one obtains

\begin{equation}
 F_1=-(x_1+x_2),\,
G_1=2x_1x_2,\, E_1=2.
\end{equation}
One easily checks that $F_1^2-E_1G_1=A$,

\begin{equation}
\aligned E_{\alpha}&=6l_1-(x_1+x_2)^2+2\alpha\\
F_{\alpha}&=2cl+x_1x_2(x_1+x_2)-\alpha(x_1+x_2)\\
G_{\alpha}&=c^2-k^2-x_1^2x_2^2+2\alpha x_1x_2.
\endaligned
\end{equation}

\medskip

Jurdjevic considered a deformation of the Kowalevski case associated
to the Kirchhoff elastic problem, see \cite{Jur}. The systems are
defined by the Hamiltonians
$$
H=\frac{1}{4}\left(M_1^2+M_2^2 + 2M_3^2\right) +\gamma_1
$$
where the deformed Poisson structures $\{\cdot, \cdot\}_{\tau}$ are
defined by
$$
\{M_i, M_j\}_{\tau}=\epsilon_{ijk}M_k,\quad \{M_i,
\gamma_j\}_{\tau}=\epsilon_{ijk}\gamma_k,\quad \{\gamma_i,
\gamma_j\}_{\tau}=\tau \epsilon_{ijk}M_k,
$$
and where the deformation parameter takes values $\tau = 0, 1, -1$.
These structures correspond to $e(3)$, $so(4)$, and $so(3,1)$
respectively.  The classical Kowalevski case corresponds to the case
$\tau=0$. The systems with $\tau = -1, 1$ have been considered by
St. Petersburg's school (Komarov, Kuznetsov \cite{KoK}) in 1990's,
and they have been rediscovered by several authors in the meantime.
Here, we are giving explicit formulae in theta-functions for the
solutions of these systems.

Denote
$$
\aligned e_1&=x_1^2-(\gamma_1 + i\gamma_2) +\tau\\
e_2&=x_2^2-(\gamma_1 - i\gamma_2) +\tau,
\endaligned
$$
where
$$
x_{1,2}=\frac{M_1\pm i M_2}{2}.
$$
The integrals of motion
$$
\aligned I_1&=e_1 e_2\\
I_2&=H\\
I_3&=\gamma_1 M_1 + \gamma_2 M_2 + \gamma_3 M_3\\
I_4&=\gamma_1^2 + \gamma_2^2+\gamma_3^2 + \tau (M_1^2+M_2^2+M_3^2)
\endaligned
$$
may be rewritten in the form (\ref{eq:firstint})
$$
\aligned k^2&=I_1=e_1 \cdot e_2\\
M_3^2&=e_1 + e_2  + \hat E(x_1,x_2)\\
-M_3\gamma_3&=-x_2e_1-x_1e_2 + \hat F(x_1, x_2)\\
\gamma_3^2&=x_2^2e_1+x_1^2e_2 + \hat G(x_1,x_2),
\endaligned
$$
where
$$
\aligned \hat G(x_1,x_2)&=-x_1^2x_2^2-2\tau x_1x_2 -2\tau I_2
+\tau^2+I_4-I_1\\
\hat F(x_1,x_2)&=(x_1x_2+\tau)(x_1+x_2)-I_3\\
\hat E(x_1,x_2)&=-(x_1+x_2)^2+2(I_2-\tau).
\endaligned
$$
\begin{proposition}
An affine  gauge transformation $ s\mapsto t +\alpha $ transforms
the Kowalevski top with $c=-1$ to the Jurdjevic elasticae according
to the formulae
$$
\tau =-\alpha,\, I_2=3l_1,\, I_3=2l,\, I_4=1-\alpha^2-6l_1 \alpha.
$$
\end{proposition}
\medskip
We can apply the generalized K\"{o}tter transformation derived in
\cite{Drag3} to obtain the expressions for $M_i,\gamma_i$ in terms
of $P_i$ and $P_{ij}$ functions for $i,j=1,2,3$. First we will
rewrite the equations of motion for Jurdjevic elasticae:
\begin{equation}\label{eq:jurdjevicsystem}
\aligned \dot M_1 &= \frac{M_2M_3}{2}\\
\dot M_2 &= -\frac{M_1M_3}{2} +\gamma_3\\
\dot M_3 &= -\gamma_2\\
\dot \gamma_1 &= -\frac{M_2\gamma_3}{2}+M_3\gamma_2\\
\dot \gamma_2 &= \frac{M_1 \gamma_2}{2} - M_3 \gamma_1 +\tau M_3\\
\dot \gamma_3 &= -\frac{M_1 \gamma_2}{2} + \frac{M_2\gamma_1}{2} -\tau M_2.
\endaligned
\end{equation}
\medskip
Now we introduce the following notation:
$$
\aligned
R(x_1,x_2) &=\hat E x_1 x_2+ \hat F (x_1+x_2) +\hat G,\\
R_1(x_1,x_2) &=\hat E \hat G - \hat F^2,\\
P(x_i) &= \hat E x_i^2 +2 \hat F x_i +\hat G, \quad i=1,\,2.
\endaligned
$$

\begin{lemma}For polynomial $\mathcal{F}(x_1,x_2,s)$ given with
$$ \mathcal{F}(x_1,x_2,s) = (x_1-x_2)^2s^2 -2R(x_1,x_2)s-R_1(x_1,x_2),$$
there exist polynomials $A(x_1,x_2,s)$, $B(x_1,x_2,s)$, $f(s)$,
$A_0(s)$ such that the following identity holds
\begin{equation}\label{eq:kotterjurdj}
\mathcal{F}(x_1,x_2,s)A_0(s)=A^2(x_1,x_2,s)+f(s)B(x_1,x_2,s).
\end{equation}
 The polynomials are defined by the formulae:
\begin{equation}
\nonumber
\aligned
A_0(s) &= 2s+2I_1-2\tau\\
f(s) &= 2s^3+2(I_1-3\tau)s^2+(-4\tau(I_1-\tau)-2I_2+4\tau^2+2I_4-4\tau I_2)s\\
& +(I_1-\tau)(-2I_1+2\tau^2+2I_4-4\tau I_2)-I_3^2+2(I_1-\tau)\tau^2\\
A(x_1,x_2,s) &= A_0(s)(x_1x_2-s)-I_3(x_1+x_2)+2\tau
(I_1-\tau)+2\tau s\\
B(x_1,x_2,s) &= (x_1+x_2)^2-2s-2I_1+2\tau.
\endaligned
\end{equation}

\end{lemma}

\medskip
Denote by $m_i$ the zeros of polynomial $f$ and
$$P_i=\sqrt{(s_1-m_i)(s_2-m_i)}\,\,\,i=1,\,2,\,3.$$
 The same way as in Corollary \ref{cor:kotterP_i} we get
\begin{equation}\label{eq:jurdjP_i}
P_i=\frac{\sqrt{A_0(m_i)}(x_1x_2-m_i)}{x_1-x_2}+\frac{-I_3(x_1+x_2)+2\tau
(I_1-\tau+m_i)}{(x_1-x_2)\sqrt{A_0(m_i)}}, \quad i=1,\,2,\,3.
\end{equation}
Put
$$
\aligned
X &= \frac{x_1 x_2}{x_1-x_2}, \quad Y = \frac{1}{x_1-x_2},\\
Z &= \frac{-I_3(x_1+x_2)+2\tau (I_1-\tau)}{x_1-x_2},\\
n_i &= A_0(m_i) = 2m_i+2I_1-2\tau, \quad i=1,\,2,\,3.
\endaligned
$$
The relations (\ref{eq:jurdjP_i}) can be rewritten  as a system of
linear equations
\begin{equation}
\nonumber
\aligned X+Y m_1 \left(\frac{2\tau}{n_1}-1\right)+\frac{Z}{n_1} & =
\frac{P_1}{\sqrt{n_1}}\\
X+Y m_2 \left(\frac{2\tau}{n_2}-1\right)+\frac{Z}{n_2} & =
\frac{P_2}{\sqrt{n_2}}\\
X+Y m_3 \left(\frac{2\tau}{n_3}-1\right)+\frac{Z}{n_3} & =
\frac{P_3}{\sqrt{n_3}}.
\endaligned
\end{equation}
The solutions of the previous system are
\begin{equation}\label{eq:jsol-X,Y,Z}
\aligned Y &= -\sum_{i=1}^{3} \frac{\sqrt{n_i}P_i}{f'(m_i)}\\
X &= -\sum_{i=1}^{3}
\frac{P_i\sqrt{n_i}}{f'(m_i)}\left(m_j+m_k+I_1-2\tau\right)\\
Z &= \sum_{i=1}^{3} \frac{2\sqrt{n_i}P_i}{f'(m_i)}\left(\frac{n_j \cdot
n_k}{4}+\tau (\tau-I_1)\right),
\endaligned
\end{equation}
with $(i,j,k)$ -  a cyclic permutation of $(1,2,3)$.\\
\medskip

Finally, we obtain
\begin{proposition} The solutions of the system of differential equations
(\ref{eq:jurdjevicsystem}) in terms of $P_i,\,P_{ij}$ functions are
 given with
$$
\aligned
M_1 &= \frac{\sum_{i=1}^{3}
\frac{2\sqrt{n_i}P_i}{f'(m_i)}(\frac{n_j \cdot
n_k}{4}+\tau (\tau-I_1))}{I_3\sum_{i=1}^{3} \frac{\sqrt{n_i}P_i}{f'(m_i)}}+\frac{2\tau(I_1-\tau)}{I_3}\\
M_2 &=-\frac{1}{i \sum_{i=1}^{3} \frac{\sqrt{n_i}P_i}{f'(m_i)}}\\
M_3 &= \frac{2i\sum_{k=1}^{3} \frac{n_k \sqrt{n_i
n_j}P_{ij}}{f'(m_k)}}{\sum_{i=1}^{3}
\frac{\sqrt{n_i}P_i}{f'(m_i)}}
\endaligned
$$
and
$$
\aligned
\gamma_1 &= I_2+\frac{1}{8}\left( \frac{\sum_{k=1}^{3} \frac{n_k
\sqrt{n_i n_j}P_{ij}}{f'(m_k)} }{\sum_{i=1}^{3}
\frac{\sqrt{n_i}P_i}{f'(m_i)}}
\right)^2-\frac{\sum_{i=1}^{3}\frac{P_i
\sqrt{n_i}}{f'(m_i)}(m_j+m_k +I_1 -2\tau)}{\sum_{i=1}^{3}\frac{P_i
\sqrt{n_i}}{f'(m_i)}}\\
\gamma_2 &= -2i\frac{(\sum_{k=1}^{3} \frac{n_k
\sqrt{n_in_j}}{f'(m_k)}\frac{P_iP_j}{2})\cdot
(\sum_{i=1}^{3}\frac{\sqrt{n_i}P_i}{f'(m_i)})}{\left(\sum_{i=1}^{3} \frac{\sqrt{n_i}P_i}{f'(m_i)}\right)^2}\\
&+2i\frac{ (\sum_{k=1}^{3}
\frac{n_k \sqrt{n_in_j}P_{ij}}{f'(m_k)})\cdot(\sum_{i=1}^{3}
\frac{\sqrt{n_i}}{f'(m_i)}
\frac{P_k P_{ik}-P_jP_{ij}}{2(m_j-m_k)})}{\left(\sum_{i=1}^{3} \frac{\sqrt{n_i}P_i}{f'(m_i)}\right)^2} \\
\gamma_3 &= \frac{\sum_{k=1}^{3} \frac{\sqrt{n_i
n_j}P_{ij}}{f'(m_k)}}{2i \sum_{i=1}^{3}
\frac{\sqrt{n_i}P_i}{f'(m_i)}}.
\endaligned
$$
\end{proposition}

\medskip

As we mentioned before, the formulae expressing $P_i, P_{ij}$ in
terms of the theta-functions are given \cite{Kow}. This gives
explicit formulae for the elasticae.

\

\thispagestyle{empty} \vspace*{20mm}

\end{document}